\documentclass{article}
 
 \usepackage{amsmath}
\usepackage{amsthm}
\usepackage{amsxtra}
\usepackage{amsfonts,amssymb}
 
\theoremstyle{plain}
\newtheorem{theorem}{\noindent Theorem}
\newtheorem{lemma}{\noindent Lemma}
\newtheorem{definition}{\noindent Definition}

\newtheorem{conjecture}{\noindent Conjecture}
\newtheorem{statement}{\noindent Proposition}

\theoremstyle{definition} 
\newtheorem{Rem}{Remark}

\newcommand{\Aut}{\operatorname{Aut}}
\newcommand{\meas}{\operatorname{meas}}
\newcommand{\Meas}{\operatorname{Meas}}
\newcommand{\dist}{\operatorname{dist}}
\newcommand{\ov}{\overline}
\newcommand{\wh}{\widehat}

\begin{document}
\author{A. M. Vershik\footnote{St.Petersburg Department of Steklov Mathematical
Institute, St.Petersburg, Russia, email: vershik\@pdmi.ras.ru.}}

\title{On classification of measurable functions of several variables}
\date{October 10, 2012}


\maketitle
\begin{abstract}
We define a normal form (called the canonical image) of an arbitrary
measurable function of several variables with respect to a natural
group of transformations; describe a new complete system of 
invariants of such a function (the system of joint distributions); and relate these
notions to the matrix distribution, another invariant of measurable
functions found earlier, which is a random matrix.
Bibliography: $7$ titles.
\end{abstract}

\section{Introduction}

In the first section of the paper, we set the problem and
introduce basic notions. The main part is the second
section, where we introduce the system of joint distributions and
prove the main lemma on the bases generated by a pure function. Using
these notions and results, we define a normal form of a measurable function of two (or
more) variables and introduce the so-called canonical image of a
function. This is an interesting, if not completely explicit,
model of functions, which is partly similar to a random process with a
set of measures as the parameter set. This notion,
which is apparently new, is related to a number of
interesting questions, which presumably provide a new method of
studying measurable functions of many variables. The third (sketchily written) section
links the previous analysis to the notion of matrix
distribution (which is a measure on matrices, or a random matrix) -- another
invariant of functions, which can be called randomization.
More exactly, this is a method of {\it classification of functions
via randomization}; it generalizes Gromov's idea, developed by the
author, on the classification of metrics on measure spaces. The
construction of a measurable function with a given matrix invariant,
announced in \cite{V1}, is based on the ergodic method, systems of
joint distributions, and the canonical image considered in Sec.~2.
The author hopes to develop these ideas in another publication.

\subsection{Setting of the problem}

Recall that the classification of measurable functions of one variable
defined on a Lebesgue space $(X,\mu)$ with a continuous measure $\mu$
and taking
values in a standard Borel space $A$ with respect to the group of all
measure-preserving transformations of the argument
is almost obvious. Namely, a measurable,  a.\,e.\ one-to-one function
$f:X\rightarrow A$ is equivalent to the identity function on the space
$A$ endowed with the measure $f_*\mu$, where 
$f_*\mu$ is the distribution of $f$ (i.e., the image of the measure
$\mu$ under $f$). This is exactly the normal (``tautological'') form
of an arbitrary measurable one-to-one function of one variable. It is
not difficult to write a normal form of an arbitrary (not necessarily
one-to-one) function. This is done in  \cite{R} (see also \cite{V1}) 
using the theory of conditional measures. To find this normal form,
one should consider the partition of the space $(X,\mu)$ into the
level sets of the function $f$ and regard  $(X,\mu)$ as a fiber bundle
over $A$. Assigning to $f_*\mu$-a.e.\ value $a \in A$ of $f$ the type 
of the conditional measure on the corresponding level set 
$f^{-1}(a)$, we turn $(X,\mu)$ into a fiber bundle defined only in
terms of $A$ and the family of these conditional measures; denote it by 
$\ov A$. Then the normal form of a general
measurable function of one variable is the projection
$F:({\ov A},{\ov \mu})\to(A,f_*\mu)$.

In this paper, we study a similar classification problem for
functions of several variables, which is far more difficult. It links
up with the paper  \cite{V1}, but here we sketch
a new method of analyzing measurable functions and relate it to the
existing approach. The method consists in considering the system of
joint distributions of sections of a function of several variables as
an invariant with respect to a natural equivalence.

Consider the product $(X\times Y,\mu\times \nu)$ of two spaces with continuous measures
and a measurable function $f$ on this
product taking values in a standard Borel space $A$.
We will always consider continuous measures
$\mu,\nu$, and the value space $A$ will always be an arbitrary standard
Borel space (finite, countable, or continual; the latter case is the
most interesting). We may assume without loss of generality that
$A$ is the interval
$[0,1]$; we will use only the Borel structure on this space.

\begin{definition} Two functions $f$ and $f'$ defined on the spaces
$(X\times Y,\mu\times \nu)$ and $(X'\times Y',\mu'\times \nu')$,
respectively, and taking values in the standard Borel space $A$ will be
called equivalent if there exist measure-preserving $(T\mu=\mu'$,
$S\nu=\nu')$
invertible transformations
$T:X\rightarrow X'$, $S:Y\rightarrow Y'$ 
such that
 $$
f'(Tx,Sy)=f(x,y),
$$
or, equivalently,
 $$ 
f'(x',y')=f(T^{-1}x',S^{-1}y').
$$
 \end{definition}

There are various special cases and generalizations of this notion,
which can be studied by the same method. In the definition above, the
group with respect to which we classify functions is the
product of the groups of measure-preserving automorphisms
of each component:
 $$\Aut(X, \mu)\times \Aut(Y, \nu).
 $$  
We mention two more possibilities.

\begin{definition}
With the above notation, assume that the factors coincide as measure
spaces: $X=Y$ and $X'=Y'$. In this case, we can consider the equivalence of
  functions with respect to automorphisms of the form $T\times T$; 
this classification will be called the {\it diagonal classification of
measurable functions}: functions $f$ and $f'$ are equivalent in this sense
if there exists an automorphism $T$ of $(X,\mu)$ such that  
  $$
  f'(Tx,Ty)=f(x,y), 
$$
or, equivalently,
  $$ f'(x',y')=f(T^{-1}x',T^{-1}y').$$
In other words, here we consider the diagonal action of the group
$\Aut(X, \mu)$.
\end{definition}

The diagonal classification is natural, for example, in the case of
functions symmetric with respect to permutations of arguments. In
particular, this setting incorporates the important problem of classification of
metrics as measurable functions on a metric measure space; this
problem is studied in detail and solved in
\cite{gr,V2}.

The second classification problem for functions of two variables is
related to the group of skew products, or the ``skew equivalence.'' It
is as follows.

\begin{definition}
For simplicity, assume that two functions
$f(\cdot,\cdot)$ 
and $f'(\cdot,\cdot)$ are defined on the same space\linebreak
$(X\times Y,\mu\times \nu)$. We say that they are skew-equivalent with
respect to the first variable if there exists a  ``skew product''
automorphism $G$ of the space $(X\times Y,\mu\times \nu)$, i.e., an automorphism of the form
 $G(x,y)=(Tx,S_x y)$
where $T:X\rightarrow X$ is an automorphism of the space
$(X,\mu)$ and $\{S_x\}_{x\in X}$ is a family of automorphisms of the space
$(Y,\nu)$, that is measurable with respect to $x\in X$ and
satisfies
 $$
f'(Tx,S_xy)=f(x,y),
$$
or, equivalently,
 $$ 
f'(x',y')=f(T^{-1}x',S_{x}^{-1}y'),\quad x=T^{-1}x'.
$$
\end{definition}

One can easily see that the skew equivalence problem can be reduced to
the first problem with one variable less and
the value space $A$ replaced with the space of probability
measures on $A$ (i.e., with the space
$\meas(A)$, or, more exactly,
$\Meas(A)$; see below). Hence it does not involve any essentially
new effects.

In a similar way we may state classification problems for
functions of more variables; we can also impose additional
restrictions (symmetries, inequalities,
etc.)\ on the values of functions, and, which is especially important, vary the groups of
automorphisms with respect to which the equivalence is being considered. The
difficulty of a classification problem depends substantially on the difficulty
of the corresponding group as a subgroup of the group of all
automorphisms. In this paper, we mainly consider the first equivalence
with respect to the direct product of the groups of automorphisms
of the variables.

In \cite{V1}, the direct and diagonal equivalence problems were solved
in terms of so-called matrix distributions, i.e.,
$S_{\infty}$-invariant measures on the space of distance matrices.
Below we give a new proof of this result, which
follows from the main theorem of this paper on the system of joint
distributions as a complete invariant.

\subsection{Basic notions}

\begin{definition}
A function $f(\cdot,\cdot)$ of two variables defined on a space
$(X\times Y,\mu\times \nu)$ is called {\it pure} if the partition 
$\xi_1$ (respectively, $\xi_2$) of the space
$(X,\mu)$ (respectively, $(Y,\nu)$) into classes of points
 $x \in X$ (respectively, $y \in Y)$ for which $f(x,\cdot)=f(x',\cdot)$ 
$\nu$-${\rm mod}~0$
 (respectively, $f(\cdot,y)=f(\cdot,y')$ $\mu$-${\rm mod}~0$) is the
 partition into points ${\rm mod}~0$.
 \end{definition}

The purety of a function of two variables is a counterpart of the
one-to-oneness of a function of one variable. The general case of
the classification problem can easily be reduced to the case of a pure function by
considering the quotient space of
$(X\times Y,\mu\times \nu)$ obtained from the partition
$\xi_1\vee \xi_2$ (see \cite {V1}). The group-theoretic meaning of the
notion of a pure function is as follows.

 \begin{statement}
A measurable function of two variables is pure if and only if its
stabilizer with respect to the action of the groups
$\Aut_Q(X,\mu)\times{\rm Id}$ and ${\rm Id}\times \Aut_Q(Y,\nu)$ is 
trivial (i.e., consists of the group identity); here $\Aut_Q(X,\mu)$
(respectively, $\Aut_Q(Y,\nu)$) is the group of transformations of the 
space $(X,\mu)$ (respectively, $(Y,\nu)$) with a quasi-invariant measure, 
which acts in a natural way (by substitutions of variables) in the space 
of measurable functions.
 \end{statement}

A stronger condition (which is more difficult to formulate in terms of
sections) is as follows.

   \begin{definition}
A function is called totally pure if its stabilizer with respect to
the action of the group
$\Aut_Q(X,\mu)\times \Aut_Q(Y,\nu)$ is trivial.
     \end{definition}

As we will see, the total purety of a function means that it has no
symmetries of the form of direct products of substitutions of
variables. Using this definition as a model, one can formulate similar purety conditions
for actions of other groups.

For the diagonal classification, the purety of $f$ means that its
stabilizer with respect to the diagonal action of the group
$\Aut_Q(X,\mu)$ is trivial. Note that, according to a theorem proved in
\cite{AuV}, the stabilizer with respect to the diagonal action of any
function that is pure in the sense of the previous definition is a
subgroup of $\Aut_Q(X,\mu)$ compact in the weak topology. It is not
difficult to formulate the notion of purety and its group-theoretic
meaning for the skew equivalence.

\section{Equivalence invariants}

\subsection{Constructing the system of coherent distributions}

In what follows, we speak only about the main equivalence,
leaving it to the reader to reformulate the results for other
equivalences.

We will define a system of characteristics that uniquely determines
a pure measurable function of several variables up to the equivalence 
described above. In fact, this is the {\it system of joint distributions}, 
whose definition is similar to Kolmogorov's definition of a random process; 
however, its peculiarity is that the parameter set (``time'') is a measure 
space, with all that this entails. Nevertheless,  for this ``pseudoprocess'' 
one can also state an analog of Kolmogorov's theorem on the existence of a 
measure on ``trajectories''; moreover, since the parameter set is a
Lebesgue space, the corresponding process is separable in spite of the
fact that the base is continual. Since we will not use this notion in what
follows, we provide almost no details, but we find it
important. One of the questions is when this definition gives a
function of several variables.

First we prepare necessary spaces. Recall that we endow the space
$$
A^n=\underbrace{A\times \cdots  \times A}_{n \ \mbox{\small{factors}}},
\quad n=1,2, \dots, 
$$ 
with the Borel structure of the product of spaces (generated by
the products of Borel sets on the factors), and by
$\meas(A^n)$ denote the simplex of all Borel probability measures
on the Borel space
$A^n$, $n=1,2 \dots$.
***
We define Borel homomorphisms $V^n_X$ and 
$V^n_Y$ of the spaces $X^n$ and $Y^n$, $n=1,2\dots$, to 
the space $\meas(A^n)$ as follows. For
$n=1$, we set $V_X(x)\equiv \alpha^f_x\in \meas(A)$ 
(respectively, $V_Y(y)=\alpha^f_y \in \meas(A)$), 
where $\alpha^f_x$ (respectively, $\alpha^f_y$) is the distribution of the section 
$f(x,\cdot)$ regarded as a function of one variable $y\in Y$ with
respect to the measure $\nu$ 
(respectively, the section $f(\cdot,y)$ regarded as a function of the
variable $x\in X$ with respect to the measure $\mu$).  The fact that
these maps are Borel follows from the measurability of the function.

Let $n>1$. To almost every, with respect to the
measure
$\mu^n\equiv\underbrace{\mu \times\cdots \times \mu}
_{n \ \mbox{\small{factors}}}$,
collection of points $(x_1,\dots, x_n) \in X^n$
(respectively, $(y_1,\dots y_n)\in Y^n$), we associate the measure 
$\alpha^{f,1}_{x_1,\dots, x_n}$ on $A^n$ that is the joint
distribution of the family of sections
$\{f(x_1,\cdot),f(x_2,\cdot), \dots, f(x_n,\cdot)\}$ regarded
as an $A^n$-valued vector function on $(Y,\nu)$; 
in the same way we define the distributions
     $\alpha^{f,2}_{y_1, \dots, y_n}$. Denote these families by
     $\Lambda^f_X=\{\alpha^f_{x_1,\dots,x_n}; \{x_i\}\in X^n,
\ n=1,2,\dots \}$      (respectively, $\Lambda^f_Y(f)$ for the second
variables); obviously, they are coherent in the natural sense:
the projection of the measure $\alpha^f_{x_1,\dots, x_n}$ on $A^n$ along the
$i$th coordinate coincides with the measure
  $\alpha^f_{x_1,\dots, \wh x_i,  \dots, x_n}$ on
$A^{n-1}$. Each of the two families $\Lambda$, regarded as a
collection of measures on   $A^n$, $n=1,2,\dots$, 
will be called the  {\it  system of joint distibutions
$(${\rm SJD}$)$} of sections of the correspondent variable; both
families together will be called the system of joint distributions of sections.
In the similar definition for functions of more than two variables, one
should consider sections of codimension one (rather than
one-dimensional sections).

\begin{Rem}
The joint distribution of sections is, obviously, well defined with
respect to ${\rm mod}\, 0$, i.e., the family $\Lambda$ depends only on the
class of functions coinciding
${\rm mod}\, 0$.
\end{Rem}

\begin{Rem}
It is obvious from the definition that the SJD is invariant under the
equivalence introduced above; more exactly, if two functions (defined
on different spaces
$(X\times Y, \mu_1\times\mu_2)$ and $(X'\times Y', \break
\mu'_1\times\mu'_2)$) are equivalent and this equivalence is implemented
by a measure-preserving transformation $T\times S: X\times Y \to X'\times Y'$, 
then the same transformation sends the SJD of the first function to
the SJD of the second function.
\end{Rem}

Fix a Borel set $B$ in $A^n$ and consider the following measurable
subsets in $X$ and $Y$:
$$ 
C^X_B=\{x\mid\exists (y_1,y_2, \dots, y_n)\in Y^n \ :
(f(x,y_1),f(x,y_2),\dots, f(x,y_n))\in B\},\, 
$$
$$
\ C^Y_B=\{y\mid\exists (x_1,x_2, \dots, x_n)\in X^n:
 (f(x_1,y),f(x_2,y),\dots, f(x_n,y))\in B\}.
$$

\begin{lemma}[on bases]
Let $f$ be a pure function of two variables. The family of sets
$\{C^X_B\}$ where $B$ ranges over a countable basis of Borel subsets in 
$A^n$ and $n$ goes from one to infinity is a countable basis of the
$\sigma$-algebra of measurable sets in
$(X,\mu)$; the same is true for the family of subsets
$\{C^Y_B\}$ in $(Y,\nu)$.
 \end{lemma}

\begin{proof}
We must prove that the $\sigma$-algebra spanned by the sets
$\{C^X_B\}$ separates points ${\rm mod}~0$ in the space $(X,\mu)$. 
But this is equivalent to the following assertion. Consider the 
equivalence relation defined as $x\sim x'\Leftrightarrow 
f(x,y_i)=f(x',y_i)$, $i=1, \dots, n$, for almost all collections
$y_1,\dots, y_n \in Y$ with respect to the measure $\nu^n$ and all 
$n=1,2, \dots$. By the ergodic theorem (more exactly, by the strong 
law of large numbers), this means that $f(x,\cdot)=f(x',\cdot)$ a.e.;  
by the purety of $f$, the latter relation determines the partition of 
$(X,\mu)$ into points ${\rm mod}~0$. Therefore, the family
$\{C^X_B\}$ generates the full $\sigma$-algebra of measurable sets
in the space $(X,\mu)$, and the product $\{C^X_B\}\times \{C^Y_{B'}\}$ 
generates the full $\sigma$-algebra of measurable sets in the space
$(X\times Y, \mu \times \nu)$. By construction, the measures of the sets
$C^X_B$ and $C^Y_B$ are determined by the SJD of $f$; hence 
a measurable function $f$ uniquely determines the measure by these data;
in turn, it is also uniquely determined by the family of
distributions of sections. In other words, if we have two measurable
pure functions $f,f'$ on the space $(X\times Y,\mu\times \nu)$ with the 
same SJD, then the functions coincide~${\rm mod}~0$.
\end{proof}

\subsection{Completeness of the system of invariants}

The next theorem states that the SJD is a complete system of
invariants of a pure function of two variables.

 \begin{theorem}[Uniqueness]
 A pure function is uniquely determined
${\rm mod}~0$ by its SJD in the following sense. Assume that two pure
functions  $f$ and $f'$ defined on spaces 
$(X\times Y, \mu\times \nu)$
and  $(X'\times Y', \mu'\times \nu')$, respectively, have the same SJD,
i.e.,
 $$
 \mu (C^X_B(f))=\mu'(C^{X'}_B(f')), \  \nu(C^Y_B(f))=\nu'(C^{Y'}_B(f'))
 $$
for all basis Borel sets $B \subset A^n$ and all $n=1,2,\dots$.
Then the functions $f$ and $f'$ are equivalent.

In particular, if two functions are defined on the same space and
their SJD coincide, then the functions coincide
${\rm mod}~0$.
  \end{theorem}
  
\begin{proof}
By the lemma, the
SJD of each of the pure functions
$f$ and $f'$ is a basis in the $\sigma$-algebra of the respective
space. Then, as mentioned at the end of the previous proof,  the correspondence between the sets
$C^X_B(f)$ and $C^{X'}_B(f')$ determines an isomorphism between the spaces  
$(X,\mu)$ and $(X',\mu')$ that sends the measure 
$\mu$ to $\mu'$. A similar assertion holds for the spaces
$(Y,\nu)$ and $(Y',\nu')$, and the product of these isomorphisms is an
isomorphism between the spaces 
$(X\times Y, \mu\times \nu)$ and $(X'\times Y', \mu'\times \nu')$ that
provides an equivalence of the functions
$f$ and~$f'$.
\end{proof}

\begin{Rem}
Writing an explicit formula for the values of a function $f$ in terms
of its SJD in the general case is a difficult task. For an
illustration, we give a concrete example of such a problem. Assume
that we are given a metric measure space; find the distance between
given two points (i.e., recover the metric) knowing all joint
distributions of the distances to any given finite collection of points.
\end{Rem}

\subsection{On systems of joint distributions and pseudoprocesses}
In fact, one can describe exactly to what extent the SJD with respect
to one variable determines the second SJD; namely, given the SJD with
respect to one
variable, a function can be recovered up to one transformation.

\begin{theorem}
Two measurable functions
$f_1$ and $f_2$ defined on a space 
$(X\times Y, \mu\times \nu)$ that have the same SJD in one (say, the
first) variable are related by the following formula: 
$$ 
f_2(x,y)=f_1(x,Ty),
$$
where $T:Y\rightarrow Y$ is a measure-preserving invertible
transformation of the space
$(Y,\nu)$.
\end{theorem}

\begin{proof}
Consider the limits of joint distributions in the variable $x$. For fixed
$x_1,x_2, \dots, x_n$, the corresponding joint distribution is a measure
$\alpha_{x_1, \dots, x_n}$ on $A^n$, i.e., on finite ordered collections of 
points from $A$. But as $n \to \infty$, we obtain a limit measure that is 
concentrated on measures (in general, continuous) on $A$. To see this,
it suffices to regard $\alpha_{x_1, \dots, x_n}$ as measures on the empirical 
distributions at the points $x_i$, $i=1,\dots, n$; then their limits will be 
distributions on $A$. But the family of these distributions determines the 
distributions of the functions $\{f(\cdot,y)\}$; more exactly, this is a 
measure on the set of distributions of $A$-valued functions in $x$; we lack 
only a correspondence between points $y \in Y$ and these distributions. Any 
two such correspondences are sent to each other by a measure-preserving 
transformation. Thus the SJD in the first variable $x$ determines a function 
$f$ of two variables up to a transformation of the second variable, i.e., a
measure-preserving transformation of the space~$(Y,\nu)$.
\end{proof}

In fact, the SJD may be regarded in the framework of the following
general definition, which does not involve any function.

 \begin{definition}
Consider a standard measure space
$(X,\mu)$ with a continuous measure and a standard Borel space~$A$. Denote
$(X,\mu)^n=(X^n,\mu^n)$, $n=1,2, \dots$.
A random pseudoprocess is a coherent system of measurable maps (more
exactly, ${\rm mod}\, 0$ classes of maps): 
$$ 
 V_n\,:\,(X,\mu)^n \rightarrow A^n,\quad n=1,2, \dots,
$$
$$
  V_n(x_1,x_2,\dots, x_n)\equiv \alpha_{(x_1,x_2,\dots, x_n)}\in \meas (A^n),
$$ 
where $\meas (\cdot)$ is the simplex of probability measures on the
corresponding Borel space; the coherence condition means that
 $$
E_{x_i}\alpha^n(x_1,\dots, x_n)=\alpha^{n-1}(x_1,\dots, 
\wh x_i,\dots, x_n),\quad n>1.
$$
 \end{definition}
In particular, given a measurable function of two variables, above we
have defined such a random pseudoprocess (in each variable).
In much the same way as a coherent system of joint distributions determines, according to
Kolmogorov, a measure on the space of trajectories, here we have
defined a pseudoprocess as a measure, but not on the set of
trajectories, but, most likely, on measurable $A$-valued functions of points of the
parameter set. Quite possibly, this notion appeared
earlier as a process on a parameter set that is a Lebesgue space.
Note that in spite of the continuality of the parameter set, the
$\sigma$-algebra on which the corresponding law is naturally defined
is separable, by the measurability of the maps
$V_n$. The properties of pseudoprocesses and the question about
whether a given pseudoprocess determines a function of two variables, or a 
polymorphism, or a more general object, will be considered elsewhere.
Here we confine ourselves to a remark anticipating the
constructions of the next section. The canonical image of a function
of two variables, which is defined below, is essentially a realization of this
function as a pseudoprocess, or a random field, over a measure
space; as such a space, one can take the set of all (one-dimensional)
distributions of sections of the function in one of the variables, and
the construction requires only knowing all joint
distributions of sections of the function.

 \subsection{Description of the model}

Now we describe a normal form for the equivalence under
consideration, which will be called
the  {\it canonical image of a pure measurable function
of two variables taking values in a Borel space $A$}.

First we construct a product space on which the normal form will be
defined. This is the square of the space
$\Meas(A)$, which we are going to define now.

Consider the simplex $\meas (A)$ of all Borel probability measures on
the Borel space $A$ endowed with the weak topology determined by a
countable basis of the $\sigma$-algebra of Borel sets.

The space $\Meas(A)$ is the (topological) direct product of the simplex
$\meas (A)$ and the compact space
$S=[0,1]\bigcup \bigcup\limits_{n=1}^{\infty}\big\{1+\frac{1}{n}\big\}$
 -- the union of an interval and a countable set of points.
It is convenient to regard
$\Meas(A)$ as the trivial fiber bundle over
$\meas (A)$ with fiber $S$ and denote by
$\pi$ the natural projection $\Meas (A)\rightarrow \meas (A)$.
The space $\Meas(A)$ is compact in the natural topology. It can be
interpreted as the space of Borel probability measures on $A$ with
multiplicities, or as the {\it space of marked measures}, the
fibers being the sets of possible marks. 

Now we describe the class of measures we will consider on this space.

\begin{definition}
A Borel probability measure on 
$\Meas(A)$ will be called {\it suitable} if it has the following form: 
its projection to  $\meas (A)$ is an arbitrary Borel probability measure 
on the simplex $\meas (A)$ (i.e., a ``measure on measures''), and the 
conditional measures on fibers $S$ turn them into general Lebesgue spaces, 
i.e., the measures on the intervals are continuous Lebesgue measures, and 
the measures of the atoms are nonincreasing in~$n$. 
\end{definition}

Note that a suitable measure is uniquely defined by its projection
to $\meas (A)$ and the types of the conditional measures on the
fibers. It is of importance that a suitable measure admits a
group of automorphisms preserving the measures on the fibers; 
this group may be trivial for some fibers or for all fibers (the 
latter is the case if the given distribution appears once as the 
distribution of a section). This group will be called the fiber group 
of a suitable measure; if a suitable measure has fibers with continuous 
components, then the fiber group is nontrivial.

\begin{definition}
A measurable function $F$ defined on the space
$\Meas(A)\times\Meas(A)$ endowed with the product
$\mu \times \nu$ of suitable measures and taking values in the
Borel space $A$ satisfies the tautology condition (or is called
tautological) if the distributions of sections of the function
$y \mapsto F(x,y)$ (respectively, $x\mapsto F(x,y)$) for almost 
all points $x\in\Meas (A)$ (respectively, $y\in \Meas (A)$), regarded 
as functions on $\Meas (A)$ of the other variable, coincide with 
the measures $\pi(x) \in \meas (A)$ (respectively, $\pi(y)\in\meas(A)$) on~$A$.
\end{definition}

\subsection{The canonical image of a function}

Using the systems of joint distributions, we present a procedure
for successively refining images of a
function so that they converge to its canonical image.

\begin{theorem}
For every pure measurable function $f$ defined on an arbitrary product of two
Lebesgue spaces with continuous measures and taking values in a Borel space
$A$ there exist unique suitable measures
$\mu_f$ and $\nu_f$ on $\Meas(A)$ and a unique measurable, pure,
tautological function
${\rm Can}\,(f)$ on the space $(\Meas(A) \times \Meas(A), \mu_f
\times \nu_f)$ such that $f$ is equivalent to 
${\rm Can}\, f$. The latter function will be called the canonical
image of $f$. The measures 
     $\mu_f$ and $\nu_f$ are invariants of the class of equivalent functions.
\end{theorem}

Thus tautological functions provide a normal form for pure functions;
in particular, the theorem says that equivalent pure functions have
the same, up to an isomorphism of fibers (see below), canonical image, and 
nonequivalent functions have different canonical images. 
Another construction of this model of functions will be described in
the last section. It is based on another approach to equivalence
invariants and differs from the construction in the proof of this
theorem in that it  automatically invokes the space
$\Meas$ and marking; more exactly, the
arbitrariness caused by automorphisms of fibers is replaced with the
arbitrariness in the choice of the initial matrix involved in the
construction of that model.

\begin{proof}
Let $f$ be an $A$-valued measurable pure function on $(X\times Y,\mu\times
\nu$).
We will denote by $f^1|_x(y)=f(x,y)$ 
(respectively, $f^2|_y(x)=f(x,y))$ the function obtained from $f$ by
fixing the
first (second) variable. Given a function $g$ on a measure space
with a measure $\gamma$, the notation $\dist (g)= g_*\gamma$ 
stands for the $g$-image of $\gamma$ on the value space, or the
distribution of $g$.

Consider the maps $L_f^1:X\rightarrow \meas(A)$
and $L_f^2:Y\rightarrow \meas(A)$ defined by the formulas
$L_1(x) = \dist (f^1|_x)$,  $L_2(y)=\dist (f^2|_y)$; that is,
the image of a point $x$ (respectively, $y$) is the distribution
of the function $f^1\big|_x$ (respectively, $f^2|_y$) regarded
as a measure on $A$ (i.e., an element of $\meas(A)$).

Each of these maps gives rise to the quotient space of
$(X,\mu)$ (or $(Y,\nu)$) constructed from the partition of this space
into the classes of points for
which the corresponding distributions coincide. These partitions are
invariant, in the sense that every equivalence preserves them.

But it obviously follows from the definition of the space
$\Meas$ that the
maps $L_1,L_2$ can be lifted (in a way that is no longer invariant) to isomorphisms of spaces
${\ov L}_1:X\to \Meas(A)$, ${\ov L}_2:Y\to \Meas(A)$; that is,
${\ov L}_f^1$ and ${\ov L}_f^2$ are isomorphisms
${\rm mod}~0$, which means that the images of different points
$x,x'$ (respectively, $y,y'$) with the same distributions
$\dist \big(f^1\big|_x\big)= \dist \big(f^1\big|_{x^{\prime}}\big)$ 
(respectively, $\dist \big(f^2\big|_y\big)= \dist \big(f^2\big|_{y^{\prime}})$\big) 
in $\Meas$ will be different ${\rm mod}~0$. 

This condition can obviously be satisfied, since
$\Meas(A)$ is a fiber bundle over $\meas (A)$ with fibers that can be turned into
an arbitratry Lebesgue space, but the isomorphism is defined only up
to fiberwise automorphisms.
The measure $\mu\times \nu$ is transferred to $\Meas(A)\times \Meas(A)$, 
and its projection to $\meas (A)$ is an invariant of the equivalence.

Fix a choice of isomorphisms and denote
$(L_f^1)_*\mu=\mu_f$ and $(L_f^2)_*\nu=\nu_f$.
Thus we have constructed an isomorphism
 $$
 L_f^1\times L_f^2: (X\times Y,\mu\times \nu)\rightarrow (\Meas(A)\times \Meas(A),\mu_f\times \nu_f).
 $$

Now consider the classes of points in
$X$ and in $Y$ with the same distributions 
$\dist (f^1|_x)= \dist (f^1|_{x'})$ (respectively,
$\dist (f^2|_y)=\dist (f^2|_{y'})$) ${\rm mod}~0$. 
With each such class $C$ we can do the following. Divide its square
$C\times C$ into classes such that inside each class the joint distribution
corresponding to the coordinates
$c_1$ and $c_2$ is the same. Thus we can refine our isomorphisms on the
classes by requiring that the classes should be preserved. Then we do
the same with the triple joint distributions, etc.
{\it It follows from Theorem~$1$ on the completeness of the SJD that
(for a pure function) the product of all such partitions converges to
the partition into points.} Hence this process yields unique
isomorphisms 
${\ov L}_1:X\to \Meas (A)$
and $ {\ov L}_2:Y\to \Meas (A)$, which will also transfer
the measures $\mu$, $\nu$ 
and the function $f$ to the image. Denote these images by $\mu_f,\nu_f$.
Thus we have obtained an invariantly defined measure space
$(\Meas (A)\times \Meas (A), \mu_f\times \nu_f)$
and a function on this space denoted by
${\rm Can}\, (f)$ and called the canonical image of $f$. The function
${\rm Can}\,(f)$ is tautological, as follows from the first step of
the construction. The uniqueness follows from the same completeness
theorem.

Applying our construction to the spaces
$X=Y=\Meas (A)$
and suitable measures $\mu$ and $\nu$ on $\Meas (A)$, and taking an
arbitrary pure tautological function on
$(\Meas(A)\times \Meas(A), \mu \times\nu)$ as $F$, one can easily
see that the functions
$L_f^1$ and $ L_f^2$ can be assumed to be identity functions. Hence
{\it every pure tautological function on the space
$(\Meas(A)\times \Meas(A), \mu \times\nu)$, where $\mu$ and $\nu$ 
are suitable measures, is a canonical image (of itself)}.
\end{proof}

Thus the class of tautological functions on the space
$\Meas (A)\times \Meas(A)$ with some suitable measures
$\mu$, $\nu$ is the class of all possible normal forms of pure
functions of two variables: every pure function is equivalent to
exactly one function from this class.

Let us make several remarks concerning the canonical image.

\medskip
1. The final space $(\Meas(A)\times \Meas(A), \mu_f \times\nu_f)$ and
the function $f$ are indeed canonical, since their constructions does
not involve objects noninvariant with respect to the equivalence.
Actually, the fact that the image is canonical follows directly from
the purety of the function, since the stabilizer in each factor is
trivial, so that the model is uniquely defined. 

\medskip

2. Consider the simplest special case. Assume that the function is
superpure, i.e., the distributions of the sections
$f(x,\cdot)$ (respectively, $f(\cdot,y)$) for
almost all points $x\in X$ (respectively, $y \in Y$) are distinct. Then the
isomorphism between $(X,\mu)$ 
(respectively, $(Y,\mu)$) and $\meas (A)$ that sends a point $x$
(respectively, $y$) to the measure on $A$ that is the distribution of
the function  $f(x,\cdot)$ (respectively, $f(\cdot,y)$) is the
required one. In this case, a basis can be constructed only from
one-dimensional distributions. To clarify the general case, one may
say that multidimensional distributions are needed in the definition
of the isomorphism only because they allow one to split distributions
having multiplicities in the SJD. However, having constructed an
isomorphism, we must give an explicit description of the function in
terms of the canonical image, and in all cases, including the simplest
one considered above, this is done via the SJD. It would be
interesting to classify measurable functions according to the minimal
collection of joint distributions needed to recover it.
\medskip

3. The space $\meas (A) \times \meas (A)$ is a topological space (the
product of simplices), as mentioned above. But the space
$\Meas (A) \times \Meas (A)$ is also a compact space, since
$\Meas = \meas\times S$. One can define a (formally) different topology
using two-dimensional, rather than only ``one-dimensional,''
distributions. For this it suffices to 
define a neighborhood of a given point as the set of
points such that the two-dimensional joint distribution corresponding
to these two points is concentrated,  in some sense, near the diagonal.

\medskip

4. If the value space $A$ is
endowed with a topology, then we can define the notion of
the {\it measurable continuity} of a function, or, more exactly, of its
canonical image. This continuity means that if
the two-dimensional distributions corresponding to two points are
close, then the values of the function at these points are also close.

\medskip\noindent
\textbf{Example.} Consider a metric $d$ on a Borel space $A$ and
define a metric on $\Meas(A)\equiv X$ by setting
$d_r(x,y)=E_{\alpha_{x_1,x_2}}r(u,v)$, where  $u,v\in A$ and the
expectation is with respect to the joint distribution
$\alpha_{x_1,x_2}$. In a similar way we define a metric on
$Y=\Meas$, and then on $X \times Y$ (the sum of the metrics on the
coordinates). We say that a function is {\it measurably continuous} if
it is continuous in this metric. This can be briefly formulated as
follows: if the mean values of the functions of one variable obtained
by fixing the second variable are close, then the values of the
functions at the given point are close. 

Presumably, an admissible metric regarded as a function of two
variables on a measure space is measurably continuous.

\begin{conjecture} A symmetric measurably continuous function of two
variables on the space
$(X \times X, \mu \times \mu)$ has a well-defined restriction to the
diagonal. More generally, for every measurable partition, every
measurably continuous function of two variables has a well-defined
restriction to almost every element of the partition. 
\end{conjecture}

\section{Relation with random matrices: classification via randomization}

\subsection{The recovery theorem}
Recall the main definitions and results of \cite{V1}. As above, let 
$f$ be a pure function of two variables $(x,y)$ defined on a space
$X\times Y$ with a product measure $\mu\times \nu$ and taking values 
in a Borel space $A$. Consider the infinite product
$(X^{\infty}\times Y^{\infty},\mu^{\infty}\times \nu^{\infty})$
with the Bernoulli measures and define a map
$$
F_f:(X^{\infty}\times Y^{\infty})\longrightarrow M_{\infty}(A),
$$
where $M_{\infty}(A)$ is the space of $A$-valued infinite matrices,
by the formula
$$
F(\{x_i\}_1^{\infty},\{y_j\}_1^{\infty})=\{f(x_i, y_j)\}_{i,j=1}^{\infty}.
$$

The map $F$ is well defined with respect to 
${\rm mod}~0$ and sends 
 $\mu^{\infty} \times \nu^{\infty}$ to a measure $D_f$,
which was called the {\it matrix distribution} of $f$. Obviously,
$D_f$ is a Borel probability measure
on $M_{\infty}(A)$ that is invariant under the group of all permutations of rows
and columns of the matrix and ergodic with respect to this group. This
follows from known facts on Bernoulli measures.

\begin{theorem}[\!\!\cite{V1}] The measure  $D_f$ is a complete
invariant of the function $f$ with respect to the equivalence; in
other words, two pure functions
$f$ and $f'$ are equivalent if and only if the measures
$D_f$ and $D_{f'}$ coincide.
\end{theorem}

Note also that the map $F$ is equivariant with respect to the square
of the symmetric group
$S_{\mathbb N}\times S_{\mathbb N}$; in the space
$M_{\infty}(A)$, this group acts by permutations of rows and columns.
The following simple assertion holds.

\begin{statement} The map $F$ is an isomorphism if and only if the
function is totally pure (see the definition in
Sec.~{\rm1}), i.e., its stabilizer in the product of groups
$\Aut_Q(X,\mu)\times \Aut_Q(Y,\nu)$ is trivial.
\end{statement}

In this case, the map $F$ yields an isomorphism between the
(Bernoulli) actions of the square of the symmetric group in these
spaces. In the case of a pure function $f$, the kernel of $F$ consists of
the orbits of the symmetry group of $f$, and, as was observed
(see \cite{AuV}), the symmetry group is compact, so that the partition
into orbits is measurable. Moreover, one can construct the quotient
space of
$(X\times Y, \mu\times\nu)$ under the partition into orbits,
but it will no longer have the structure of a direct product.
Curiously enough, in spite of the fact that 
for pure, but not totally pure, functions,
$F_f$ is not an isomorphism of the spaces
$(X^{\infty}\times Y^{\infty}, \mu^{\infty}\times \nu^{\infty})$ and $(M_{\infty}(A),D_f)$,
nevertheless, as one can see from the proof of the recovery theorem and the
lemma on bases, a basis of the space
$(X\times Y, \mu \times \nu)$ can be recovered from the basis of the
image  $(M_{\infty}(A),D_f)$.

The construction below gives, in particular, another proof of the
recovery theorem; it uses Theorem~1 and clarifies the ideas
suggested in \cite{V1}.

With obvious modifications, we can also state a similar theorem on
the diagonal equivalence. In this case, we should take one sequence
$\{x_i\}_1^{\infty}$ of independent random variables distributed on
$X=Y$ according to the measure $\mu=\nu$.
If $(X,\mu)$ is a metric measure space and 
$f$ is a metric regarded as a measurable function on
($X\times X, \mu \times \mu$) (which is obviously pure),
this is the theorem saying that the corresponding measure on the space
of metrics on the set of positive integers is a complete invariant of 
metric triples (metric measure spaces). This was proved in
\cite{gr} in a somewhat different form, and then in 
\cite{V2}, by a simpler (ergodic) method, in this form.

Thus we have two types of complete invariants of the equivalence of
arbitrary pure functions of two variables, and thus two methods of
classification of measurable functions. The first one is provided by the
{\it matrix distribution}, and it can be called the {\it
classification via randomization}, since the invariant is a randomized
net of values of the function. Although the definition of this
invariant is very simple, it does not provide a model of the function,
i.e., a method of recovering the function from this invariant, a random
matrix. Such a recovery procedure will be described below. The
second method, considered in this paper, is more cumbersome; it is the
classification via the system of joint distributions, and it allows
one to describe 
a normal form (the canonical image of a function), which is a model of the
function.

In the next section, we will combine both methods; namely, we will
show how one can construct the canonical image from the matrix
distribution, i.e., give an explicit construction of a function with a
given invariant. The construction relies on the {\it ergodic method}
\cite{V0}; more exactly, we systematically use the ergodic theorem
for an action of the symmetric group, or, more specifically,
successively construct empirical distributions.
Another version of constructing a realization of a measurable function
from its matrix distribution was considered in the unpublished
manuscript   \cite{VH}. In terms of the construction below, it uses
only the first averaging process, and hence is also less cumbersome, but
ignores the system of joint distributions.

\subsection{Constructing a function from its matrix distribution}
Let $D$ be a given Borel probability measure on the set of infinite
matrices  $M_{\mathbb N}(A)$, which is the matrix distribution of a
measurable function taking values in a Borel space 
$A$ and defined on a space
$(X\times Y, \mu\times \nu)$; we will not use this space in any way.
The construction below uses only the fact that the measure $D$ is
invariant and ergodic with respect to the action of the group
$S_{\mathbb N}\times S_{\mathbb N}$ (the square of the symmetric
group) by permutations of rows and columns. The construction proceeds
as follows (this is just an application of the ergodic method). We choose a ``typical''
matrix with respect to the measure $D$ (i.e., almost every matrix with
respect to this measure) and construct a space, a measure on this
space, and a required function.

Only then we use the fact that $D$ is the matrix distribution of
a (totally pure) function and give a direct description of
measures $D$ (called simple in \cite{V1}) that can be matrix
distributions. In this case, we will see that $D$ is exactly the
matrix distribution of  the constructed
function (and hence the latter is equivalent to the original function).
Thus we will construct the canonical image of a function with a
given matrix distribution, and the relation to the previous section is
that we construct the SJD of the function. The
construction proceeds in several steps.

\smallskip

1. Denote the matrix by $R=\{r_{i,j}\}_{i,j=1}^{\infty}$. 
By the invariance of the measure $D$ with respect to the group of permutations of
columns, we can apply the ergodic theorem to the action of the group
$S_{\mathbb N}$, considering the ``empirical distribution of rows'':
 $$
\lim_{n\to \infty} \frac{1}{n!}\sum_{g\in S_n} \chi_B(\{r_{i,g^{-1}(1)}\}) \equiv m_R(B).
$$
Here $B$ is a cylinder set in $A^{\infty}$ and 
$\chi_B$ is its characteristic function; the formula counts how
frequently the first column falls into $B$ when the matrix is being permuted. Thus we
obtain a probability measure   $m_R$ on the space
 $A^{\mathbb N}$; of course, it depends on the matrix $R$.

\smallskip
 2. Using the invariance of $D$ with respect to the group
 of permutations of rows, we find the ``empirical distribution of
 measures.'' In more detail, we lift the permutational action of the
 group $S_{\mathbb N}$ on rows from the space
$A^{\mathbb N}$ to the simplex $\meas (A^{\mathbb N})$
of measures on this space,  consider the orbit of the constructed measure
$m_R$ on $A^{\mathbb N}$ with respect to this action, and find the
empirical distribution of measures on this orbit. For this we should
consider a basis in the space of Borel sets on the simplex of measures.
Let $\mathcal B$ be a Borel set in the space of measures. Then the
empirical measure of this set (of measures on $A^{\mathbb N}$) is
given by the formula
 $$
 \lim_{n\to \infty} \frac{1}{n!}\sum_{g\in S_n} \chi_{\mathcal  B}(g(m_R))\equiv M_R(B).
 $$
In this formula, the action of an element $g \in S_n$ is applied to
the measure $m_R$. The limit exists by the same reasons for
$D$-almost all matrices $R$; in terms of $R$, this is a double limit,
and its existence follows from the invariance of $D$ with respect to
the second (row) component of the group
$S_{\mathbb N}\times S_{\mathbb N}$. We have constructed a measure on
$\meas (A^{\mathbb N}) $ (i.e., a measure on measures on $A^{\mathbb N}$)
induced by the measure $m_R$ and, indirectly, by the matrix $R$. 
Denote it by $M_R$.

\smallskip

3. Now we do the same in the other order. First, using the same matrix
$R$, we construct the empirical measure $m^R$ on the columns (using the 
invariance under the group of permutations of rows); this is again a 
measure on~$A^{\infty}$.

\medskip
4. Then, as at Step 2, we take the average of 
$m^R$ over the columns to obtain a measure on
$\meas (A^{\mathbb N})$, i.e., a ``measure on measures''
$M^R$ similar to $M_R$.

\medskip
5. The next important assertion is that each of the measures
$M_R$ and $M^R$, which were initially defined as measures on 
$\meas (A^{\infty})$, induces a pair of measures on
$\Meas A^n$ for all $n=1,2, \dots$, namely, the finite-dimensional
(with respect to $\meas (A^{\infty})$) empirical distributions on
measures on   $A^n$; we denote them by $\mu^n_R$ and $\nu^n_R$. 
For $n=1$, they play exactly the same role as the measures
$\mu_f$ and $\nu_f$ on $\Meas A$, which were introduced in 
Sec.~2 in the definition of the canonical image of a function. Using
these measures, we obtain the SJD of the future function and,
consequently, the function itself. This is the transition to the
canonical image. We do not describe what systems of joint
distributions determine measurable functions. In the previous section,
we moved in the opposite direction and constructed the SJD from a
given function. One can show that the measures
$M_R$ and $M^R$ also determine the function, but we do not use this fact
below, since the existence of the function is already guaranteed and
the problem is how to realize it. Denote this function on
$\Meas A \times \Meas A$ by $f_R$; it is the goal of our construction.
As yet, we have used only the invariance of the measure $D$ with
respect to the square of the symmetric group.

Now we state the main results on this function and on the whole
construction. First we define a class of measures on the space of
matrices specified by a condition called the simplicity of $D$.

\begin{definition} A measure $D$ on the space of matrices $M_{\mathbb
N}(A)$ that is invariant and ergodic with respect to the square of the
symmetric group is called simple if the procedure described above that
with a matrix $R \in (M_{\mathbb N}(A),D)$ associates the pair of
measures
$$
(m_R,m^R) \in  \Meas(A^{\infty})\times \Meas(A^{\infty})
$$
is an isomorphism of measure spaces
$$
(M_{\mathbb N}(A),D) \rightarrow  
(\Meas(A^{\infty})\times \Meas(A^{\infty}), M_R\times M^R).
$$
\end{definition}

The difference between pure and totally pure functions in the context
of matrix distributions was observed by U.~Hab\"ock in
2005; it is not mentioned in the original paper
\cite{V1}, which caused the appearance in that paper of a false assertion 
that the matrix distributions of pure functions are simple measures; it was
corrected in \cite{VH}. This difference forces us to define not only
simple measures, but also {\it semisimple measures}, which are the matrix
distributions of simple functions. A direct description of semisimple
measures in the same terms as in the case of simple measures is hardly possible, 
the reason being that taking the quotient with respect to the compact
subgroup of symmetries
of a function violates the structure of a direct product, so that the
described procedure yields a function defined on a more complicated object,
since its variables are no longer independent. 
But the case of simple measures and, accordingly, of totally pure
functions, is exhausted as follows.

     \begin{theorem}
{\rm1.} If the measure $D$ is ergodic, then the measures $M_R$ and
$M^R$ on $\meas (A^{\infty})$ and, consequently, the measures
$\mu_R,\nu_R$ and the function $f_R$ constructed above coincide for
$D$-almost all matrices $R$. If 
$D=D_f$ is the matrix distribution of a totally pure function $f$,
then the constructed function
$f_R$ is equivalent to $f$, and its matrix distribution coincides with
$D$. Thus the canonical image of
$f_R$ is a realization of the function with a given matrix distribution.

\smallskip
     {\rm2.} The class of simple measures coincides with the class of
     matrix distributions
$D_f$ of totally pure measurable functions $f$.
\end{theorem}

\begin{proof}
The first assertion in Claim 1 follows immediately from the ergodicity
of the action of the group
$S_{\mathbb N}\times S_{\mathbb N}$. The remaining part of Claim~1 is
reduced to more detailed comments on the construction, which is a
proof. Comparing it with the construction of the canonical image shows
that the required function is constructed via the SJD, as we have
already mentioned.

Now we proceed to Claim 2. Assume that we have a totally pure function
and we have chosen a typical realization of a sequence of i.i.d.\
variables $x_1,x_2, \dots; y_1,y_2,\dots$. Consider the matrix
$R=\{f(x_i,y_j)\}$ and keep track of the procedure described in
Sec.~3.2. The first step, finding the measure $m_R$, is equivalent to 
computing the joint distributions of the countably many functions
$f(x_1,\cdot),f(x_2,\cdot),\dots$ regarded as functions of the second
variable. The measure $m_R$ on $A^{\infty}$ is exactly the joint distribution 
of these functions. The second step, computing the empirical measure from
$\{x_i\}_i$, determines the measure $M_R$ (a measure on measures), but it 
is nothing else than a measure on the distributions of the functions 
$f(\cdot,y)$; more exactly, this is the image of the measure
$(Y,\nu)$ under the map
$y\mapsto \dist_x [f(\cdot,y)]$, where $\dist_x$ stands for the
distribution of a function of $x$ with a fixed $y$. The second series
of operations with the roles of rows and columns interchanged is
interpreted in the same way. Finally, as we have already observed, in the case of a totally pure function the
map $F_f$ is an isomorphism, so
that the map from the space of matrices endowed with the measure
$D_f$ to the space of empirical measures in the definition of a simple
measure is also an isomorphism.

The converse obviously follows from the construction. Namely, consider
a simple measure $D$ on
$M_{\mathbb N}(A)$; by the isomorphism of spaces in the definition of
a simple measure, we do not lose information when constructing the
function $f_R=f$, and the obtained function is totally pure, since the
map $F_{f_R}$ is also an isomorphism. To see that the matrix
distribution of this function is the measure $D$, one can look through
the construction of $f_R$ from the SJD: in fact, the
finite-dimensional restrictions 
of the measures
$M_R$ and $M^R$ provide an approximation of the matrix distribution
by finite-dimensional distributions.
\end{proof}

One can easily see that the described construction is an exact
reproduction of the procedure describing the construction of the
canonical image of a function; it is even somewhat more formalized,
since here the space
$\Meas$ (accounting for the multiplicities of distributions) appears
automatically. Thus, given a matrix $R$, we have constructed the SJD,
which, according to the theorem on bases, 
allows one to recover the
function and hence the matrix distribution.

\smallskip
Supported by the RFBR grants
11-01-12092-ofi-m and 11-01-00677-a.

\end{document}